\documentclass[11pt,a4paper,reqno]{amsart}
\usepackage[T1]{fontenc}
\usepackage{lmodern}
\usepackage[a4paper,margin=1in]{geometry}
\usepackage{amsmath,amssymb,amsthm,mathtools}
\usepackage{mathrsfs,enumitem,microtype}
\usepackage{tikz-cd}
\usepackage{cite}
\usepackage{xcolor}
\definecolor{referenceblue}{rgb}{0,0,0.55}
\definecolor{referencegreen}{rgb}{0,0.4,0}
\definecolor{referenceurl}{rgb}{0,0,0.65}
\usepackage[colorlinks=true,linktoc=section,
  linkcolor=referenceblue,citecolor=referencegreen,
  urlcolor=referenceurl]{hyperref}
\allowdisplaybreaks[1]
\newtheorem{theorem}{Theorem}[section]
\newtheorem{proposition}[theorem]{Proposition}
\newtheorem{lemma}[theorem]{Lemma}

\theoremstyle{definition}

\newcommand{\Z}{\mathbb Z}

\newcommand{\R}{\mathbb R}
\newcommand{\CC}{\mathbb C}
\newcommand{\PP}{\mathbb P}
\newcommand{\cO}{\mathcal O}
\DeclareMathOperator{\MW}{MW}
\DeclareMathOperator{\im}{im}

\DeclareMathOperator{\rk}{rank}
\DeclareMathOperator{\specialize}{sp}
\DeclareMathOperator{\Pic}{Pic}

\title[A complex rational homology six-sphere]
{A complex structure on a rational homology six-sphere with two-torsion}
\author{Zichang Wang}
\date{}
\address{
Z.C. ~ Wang: Tsinghua University, Beijing, China
}
\email{zichang-21@mails.tsinghua.edu.cn}
\begin{document}
\begin{abstract}
We construct a compact, smooth, simply connected complex threefold.
Its second and third integral homology groups are cyclic of order two,
and its remaining homology groups agree with those of the six-sphere.
Starting from a rational elliptic surface with singular fibers
\([IV^*,I_2,I_1,I_1]\), we form a family of complex two-tori and
complete it using the Mumford construction and a logarithmic
transformation of multiplicity three. We compute the fundamental group
and integral homology of the resulting threefold.
\end{abstract}
\maketitle
\tableofcontents

\section{Introduction}\label{intro:section}

Which smooth manifolds admit complex structures? A central example is
the classical Hopf problem, which asks whether the six-sphere admits
an integrable complex structure.
Borel and Serre showed that the only positive-dimensional spheres
admitting almost complex structures are $S^2$ and $S^6$~\cite{BorelSerre}.
Albanese and Milivojevi\'c extended this restriction to rational
homology spheres: an almost complex rational homology sphere has real
dimension two or six~\cite[Theorem~2.2]{AlbaneseMilivojevic}.
Thus six is the only possible dimension beyond the Riemann sphere
for a compact complex manifold with the rational homology of a sphere.

The manuscript \emph{A compact complex threefold fibred by tori over
the projective line, and the six-sphere}, attributed to Claude and
Levent Alp\"oge~\cite{AlpogeClaude}, presents a construction of a
complex structure on $S^6$. It starts with a family of complex
two-tori over a three-punctured projective line, uses the Mumford
construction at the point with unipotent monodromy, and applies
logarithmic transformations of multiplicities three and four at
the other two points. Engel's short note~\cite{Engel} explains this
construction using a rational elliptic surface and a quotient of a
principal $\CC^*$-bundle by a lifted Mordell--Weil translation.

We use a similar construction, starting from a rational elliptic
surface with singular fibers $[IV^*,I_2,I_1,I_1]$. If $O$ is the
zero section and $P$ generates the Mordell--Weil group, we set
$M=\cO_S(P-O)$. Over the smooth locus of the elliptic fibration,
we quotient $M^\times$ by a sufficiently small scalar multiple of
a lift of translation by $-6P$. We complete the
resulting torus family using the Mumford construction at the three
multiplicative fibers and a logarithmic transformation of multiplicity
three at the $IV^*$ fiber. Our proof does not rely on the conclusions of~\cite{AlpogeClaude}. The result is the following.

\begin{theorem}\label{intro:main}
For every sufficiently small nonzero $\lambda$, the construction gives
a compact, smooth, simply connected complex threefold $X_2(\lambda)$
with integral homology
\[
 H_k(X_2(\lambda);\Z)\simeq
 \begin{cases}
 \Z,&k=0,6,\\
 \Z/2,&k=2,3,\\
 0,&k=1,4,5.
 \end{cases}
\]
In particular, $X_2(\lambda)$ is a complex rational homology six-sphere.
\end{theorem}

The theorem is proved in Section~\ref{hom:section}. Its nonzero
integral torsion distinguishes the underlying manifold from $S^6$.
The essential point is the Mumford construction at the $I_2$
fiber: the exponent vector $(2,1)$ of the multiplicative gluing is
primitive, allowing us to choose a model whose central fiber has
one irreducible component. We prove that specialization
for all three Mumford models maps onto the full integral invariant
lattice. At the remaining fiber, we normalize the identification
after the cubic base change to preserve the section $e$
of Proposition~\ref{init:section}
before performing the logarithmic transformation. This fixes both
the meridian relation and the translation class in the Leray
spectral sequence. The four transgressions have coefficients
$(1,2,2,1)$, giving the two copies of $\Z/2$ above. The same local
identifications show that the fundamental group is trivial.

Section~\ref{init:initial} constructs the open torus family and
completes it using the Mumford construction and a logarithmic
transformation. Section~\ref{main:section} specifies the global
threefold. Sections~\ref{mon:section}, \ref{pi:section}, and
\ref{hom:section} compute its monodromy, fundamental group, and
integral homology. Section~\ref{top:section} identifies the underlying
oriented smooth manifold as surgery on a conic in $\PP^3$.
Appendix~\ref{app:elliptic} contains the standard elliptic-surface
calculations, and Appendix~\ref{app:topology} proves the topological
classification.

\subsection*{Acknowledgements}
The author would like to thank Shing-Tung Yau for helpful discussion and Philip Engel for his short note~\cite{Engel},
which was instrumental in understanding the construction
of~\cite{AlpogeClaude}.

\subsection*{AI Statement}
GPT 5.6 Sol was used for the calculations. GPT 6 Astra was
used to improve the grammar and exposition of this article. All arguments and calculations have been verified by the author, who takes responsibility for the mathematical content.

\section{Construction}
\label{init:initial}

We use the notation of Engel~\cite[Section~2]{Engel}: $\pi:S\to\PP^1$
is a rational elliptic surface, $O$ its zero section, and $P$ a
Mordell--Weil generator. The surface used here has an $I_2$ fiber and
two $I_1$ fibers in place of Engel's $III$ and $I_1$ fibers. The
line bundle $M=\cO_S(P-O)$ has degree zero on each smooth fiber;
the line bundle governing its linearization is the pullback of
$\cO_{\PP^1}(1)$. We collect the required facts below and give their
proofs in Appendix~\ref{app:elliptic}.

\begin{proposition}\label{init:surface}
The minimal regular model of
\begin{equation}\label{init:weierstrass}
 y^2=x^3-3(1+t)x+(2+3t+2t^2)
\end{equation}
is a rational elliptic surface $\pi:S\to\PP^1$ with zero section $O$.
Its singular fibers at $\infty,0,\alpha_+,\alpha_-$ have types
$IV^*,I_2,I_1,I_1$, respectively, where
\[
 \{\alpha_+,\alpha_-\}=\{-1+i/2,-1-i/2\}.
\]
\end{proposition}

The marking in Section~\ref{mon:section} fixes the labels
$\alpha_+,\alpha_-$.

\begin{lemma}\label{init:height}
The Mordell--Weil group is infinite cyclic, generated by
\begin{equation}\label{init:explicit-section}
 P(t)=(1,\sqrt2\,t).
\end{equation}
Moreover, $P\cdot O=0$ and $\langle P,P\rangle=1/6$.
\end{lemma}

Set $M=\cO_S(P-O)$. Translation by $6P$ extends to an
automorphism $t_{6P}$ of $S$. Since the component groups at $IV^*$ and
$I_2$ have orders three and two, ${6P}$ meets the identity component of
every fiber. Thus $t_{6P}$ preserves each fiber component.

\begin{lemma}\label{init:linearization}
There is an isomorphism
\begin{equation}\label{init:linearization-line}
 \mathcal{H}om(t_{6P}^*M,M)\simeq\pi^*\cO_{\PP^1}(1).
\end{equation}
\end{lemma}

Choose a morphism \(\Phi_0:t_{6P}^*M\to M\) whose zero divisor is
exactly \(S_{\alpha_-}\).  It exists by
\eqref{init:linearization-line}, and is unique up to a nonzero scalar.
Set
\[
 U=\PP^1\setminus\{\infty,0,\alpha_+,\alpha_-\},\qquad
 S_U=\pi^{-1}(U),\qquad \Phi=\lambda\Phi_0,
\]
where \(\lambda\in\CC^*\) will be chosen sufficiently small.
Write \(M^\times\) for the complement of the zero section in \(M\). And write \(M^\times_U\) for \(M^\times|_{S_U}\).
The lifted translation generating the cyclic action is
\begin{equation}\label{init:generator}
 F(v_x)=\Phi_{x-{6P}}(v_x)\in M_{x-{6P}}^\times,
 \qquad x\in S_U,\quad v_x\in M_x^\times.
\end{equation}
Indeed, \((t_{6P}^*M)_{x-{6P}}=M_x\); thus \(F\) covers translation by
\(-{6P}\). We use this generator throughout, including in the period
marking of Section~\ref{mon:section}.

\begin{proposition}\label{init:open}
For every sufficiently small \(\lambda\ne0\), the action generated
by \(F\) on \(M^\times_U\) is free and properly discontinuous.
Its quotient
\begin{equation}\label{init:open-family}
 X_U=M^\times_U/\langle F\rangle\longrightarrow U
\end{equation}
is a proper holomorphic submersion with compact complex two-torus fibers.
\end{proposition}

\begin{proof}
Choose a Hermitian metric on $M$. Compactness of $S$ gives
$\|\Phi_0\|\leq C$; take $c=|\lambda|C<1$. On $S_U$, the lift
$F$ is invertible and $\|F(v)\|\leq c\|v\|$. Positive iterates
contract and negative iterates expand, so the action is free.
A compact set has norms bounded above and away from zero, and thus
meets only finitely many translates. The quotient is a complex
manifold, and its map to $U$ is a submersion.

For compact $K\subset U$, invertibility gives a uniform lower bound
$\|F(v)\|\geq c_K\|v\|$ on $S_K=\pi^{-1}(K)$, with $c_K>0$. Every orbit
meets the compact annulus
\[
 \{v\in M|_{S_K}:c_K\leq\|v\|\leq1\},
\]
whose image is the inverse image of $K$ in $X_U$. Thus the map is proper.

Rigidifying $M_t$ at $O(t)$ gives a semiabelian group
\[
 1\longrightarrow\CC^*\longrightarrow M_t^\times
 \longrightarrow S_t\longrightarrow0.
\]
The lift $F$ differs from a group translation over $-{6P}(t)$ by a
nowhere-zero holomorphic function on $S_t$, hence by a constant.
It is therefore a translation. The fiber quotient is a compact
connected commutative complex Lie group of dimension two, hence a
complex two-torus.
\end{proof}

Fix such a \(\lambda\).

\begin{proposition}\label{init:section}
There is a section \(e\) of \(M|_O\) with divisor \(\alpha_-\).
Fix such an \(e\). By abuse of notation, we also write \(e\) for the
induced holomorphic section \(t\mapsto[e(t)]\) of \(X_U\to U\).
This section extends across the rank-one fillings at \(0,\alpha_+\) and the
rank-two filling at \(\alpha_-\) constructed below.
\end{proposition}

\begin{proof}
Since \(O\simeq\PP^1\) and \(\deg(M|_O)=(P-O)\cdot O=1\), a section
with divisor \(\alpha_-\) exists. For \(t\in U\), the vector
\(e(t)\in M_{O(t)}^\times\) fixes the identity in the semiabelian
group \(M_t^\times\), so its class is the identity in the quotient \(X_t\).
At \(0\) and \(\alpha_+\), the representative of \(e\) has valuation zero, so the section lies in the
zero-vertex chart of the rank-one models of
Proposition~\ref{fill:rankone}.  At \(\alpha_-\), both \(e\) and
\(\Phi\) vanish simply.  In the multiplicative coordinates used in
Proposition~\ref{fill:a2}, the period valuations form a basis of
\(\Z^2\), and the representative \(e\) has valuation \((0,1)\).
This is a vertex of the periodic unimodular triangulation, hence
\(e\) extends in its toric chart.  These extension statements are
verified together with the respective local models.
\end{proof}

We extend $X_U\to U$ across the four missing points using the
Mumford construction at $0,\alpha_+,\alpha_-$ and a logarithmic
transformation at infinity. At $0$, the valuation vector $(2,1)$
is primitive, so we can choose a Mumford model whose central fiber
has just one irreducible component, despite the two components of
the original $I_2$ fiber. At infinity, we first pass to the cubic
cover $s=r^3$ and then perform a logarithmic transformation of
multiplicity three. For each filling, we specify the gluing map
over the punctured disc, as these maps are needed to compute the
topology of the resulting threefold.

\subsection{The rank-one Mumford construction}\label{fill:rankone-section}

We describe the two rank-one degenerations directly, using the
multiplicative uniformization of elliptic curves as in
\cite[Proposition~2.7]{Engel}. At these two points the linearization
is invertible, so only one direction degenerates. The exponent vector
of a gluing map records the orders of its two multipliers in the local
parameter.

\begin{lemma}\label{fill:tate}
At each of $0$ and $\alpha_+$, let $q$ be a local parameter vanishing
at that point. The punctured torus family has a biholomorphic presentation
\begin{equation}
 \begin{gathered}
 X_q\simeq(\CC^*)^2/\langle h,g\rangle,\\
 h(z,v)=(h_1z,h_2v),\qquad g(z,v)=(c_3z,\lambda c_4v),
 \end{gathered}
 \label{fill:tate-period}
\end{equation}
where
\[
 (h_1,h_2)=
 \begin{cases}
 (c_1q^2,c_2q),&\text{at }0,\\
 (c_1q,c_2),&\text{at }\alpha_+,
 \end{cases}
\]
and $c_i=c_i(q)\in\mathbb{C}^*+O(q)$ are holomorphic units.

For the lift of $F$ chosen in the proof, $g=Fh^3$ at $0$ and
$g=F$ at $\alpha_+$. The exponent vector of $g$ is $(0,0)$;
that of $h$ is $(2,1)$ at $0$ and $(1,0)$ at $\alpha_+$.
Both nonzero vectors are primitive.
\end{lemma}

\begin{proof}
Write $S_q=\CC^*/h_1^{\Z}$ and represent $P$ by $h_2$.
The multiplicative description of the degree-zero line bundle
$M_q$~\cite{BL} gives the gluing $(z,v)\sim(h_1z,h_2v)$.
Thus $M_q^\times=(\CC^*)^2/\langle h\rangle$.

At $0$, $h_1$ has order two, and we choose $h_2$ of order one
because $P$ meets the nonidentity component of the $I_2$ fiber
(Lemma~\ref{init:height}). At $\alpha_+$, $h_1$ has order one
and $h_2$ can be chosen to be a unit. This gives the stated formulas.

Since $F$ covers translation by $-6P$, choose a lift to $(\CC^*)^2$
whose first multiplier is $h_2^{-6}$. At $0$, this multiplier has
a pole of order six. Replacing the lift by $g=Fh^3$ cancels the pole:
\[
 h_1^3h_2^{-6}=(c_1q^2)^3(c_2q)^{-6}=c_1^3c_2^{-6}.
\]
This leaves the quotient unchanged, since
$\langle h,Fh^3\rangle=\langle h,F\rangle$.
At $\alpha_+$, $h_2^{-6}=c_2^{-6}$ is already a unit, so $g=F$.

In both cases, $g$ lifts translation by $-6P$ using a unit
representative on the identity component. The linearization
$\Phi_0$ extends invertibly, so its fiber multiplier in a regular
frame is a holomorphic unit. Multiplication by $\lambda$ gives
the second multiplier $\lambda c_4$. Hence $g$ has exponent vector
$(0,0)$. The vectors $(2,1)$ and $(1,0)$ for $h$ are primitive.
\end{proof}

We use the following form of the periodic fan construction
\cite{Mumford,Engel}. A period
\[
 (z_1,z_2,q)\longmapsto
 (q^{b_1}u_1(q)z_1,q^{b_2}u_2(q)z_2,q),
 \qquad u_i\in\cO_\Delta^{\times},
\]
acts on the valuation lattice by
$(a,m)\mapsto(a+m(b_1,b_2),m)$. A regular periodic subdivision of
the height-one slice gives smooth toric charts; the unit factors act
on these charts by the torus action. In rank two, a complete
subdivision with compact quotient by the period lattice gives a
proper analytic model. In rank one, the period with zero exponent
vector first compactifies the remaining $\CC^*$ direction to an
elliptic curve. Local finiteness of the charts and a compact annular
domain for this period ensure that the quotient is Hausdorff.

\begin{proposition}
\label{fill:rankone}
At the $0$ and at $\alpha_+$, the punctured torus family has
a proper filling with smooth total space and a one-component
rank-one semiabelic central fiber $W$. The section $e$
extends across this filling, and specialization induces integral isomorphisms
\[
 \specialize^k:H^k(W;\Z)\xrightarrow{\sim}H^k(X_q;\Z)^T,
 \qquad 0\leq k\leq4.
\]
\end{proposition}

\begin{proof}
Use the generators of Lemma~\ref{fill:tate}. At $0$, change coordinates by
\[
 x=v,\qquad y=v^2/z.
\]
The exponent matrix has determinant one, so this is an invertible
monomial change. It gives
\[
 h(x,y)=(qc_2x,c_2^2c_1^{-1}y).
\]
Thus the exponent vector $(2,1)$ becomes $(1,0)$. At $\alpha_+$,
the original coordinates $(x,y)=(z,v)$ already have this property.
The multiplier of $g$ in the $y$ direction is
\[
 c=
 \begin{cases}
 c_3^{-1}(\lambda c_4)^2,&\text{at }0,\\
 \lambda c_4,&\text{at }\alpha_+.
 \end{cases}
\]
After shrinking the discs and decreasing $|\lambda|$, we have
$0<|c|<1$ uniformly. The quotient $\CC^*/c^{\Z}$ defines
a smooth elliptic family $E\to\Delta$; a closed annulus gives a
compact fundamental domain.

For the degenerating direction, use the rays and cones
\[
 \rho_k=\R_{\geq0}(k,1),\qquad
 \sigma_k=\R_{\geq0}(k,1)+\R_{\geq0}(k+1,1),
 \quad k\in\Z.
\]
Every cone is regular. The period $h$ shifts $k$ to $k+1$, so all
components lie in one orbit. Quotienting first by $g$ makes each
component a $\PP^1$-bundle over $E$; quotienting by $h$ then identifies
consecutive boundary sections. Neither action has a boundary
stabilizer. The complete subdivision and properness of $E$ give a
proper quotient. Along the double curve, its local equation is $xy=q$.

The normalization and gluing of $W$ are
\[
 \overline W=\PP_E(\cO_E\oplus L),\quad L\in\operatorname{Pic}^0(E),
 \qquad E_0\sim E_\infty\text{ by a translation }\tau.
\]
Since $c_1(L)=0$ and $\tau$ is isotopic to the identity, the
normalization sequence has zero difference maps on integral
cohomology. It gives, noncanonically,
\begin{equation}
 H^k(W;\Z)\simeq H^k(E\times\PP^1;\Z)\oplus H^{k-1}(E;\Z).
 \label{fill:rankone-cohomology}
\end{equation}
In particular these groups are torsion-free.

To determine the index of specialization, trivialize $L$
topologically and isotope the boundary translation to the identity.
The normal bundles of the double curve are topologically trivial.
Let $\kappa_q:X_q\to W$ be the continuous map that collapses the
vanishing circle over each point of the double curve in the charts
$xy=q$ and uses the chosen trivialization away from their collars.
Its pullback induces specialization, and the following diagram
commutes up to homotopy.
\begin{equation}
\begin{tikzcd}[column sep=large, row sep=large]
 X_q \arrow[r,"\kappa_q"] \arrow[d,"\simeq"']
   & W \arrow[d,"\simeq"] \\
 E\times T^2 \arrow[r,"\mathrm{id}_E\times\kappa"']
   & E\times(S^2\vee S^1).
 \end{tikzcd}
 \label{fill:rankone-collapse}
\end{equation}
Here $\kappa$ pinches one primitive circle in $T^2$.
On the standard cellular chains, $\kappa_*(a)=0$,
$\kappa_*(b)=\beta$, and $\kappa_*[T^2]=[S^2]$.

Choose integral cohomology generators so that
\[
 \begin{gathered}
 T(e_3)=e_3+e_0,\quad T(e_i)=e_i\ (i=0,1,2),\\
 \kappa^*(\beta^*)=e_0,\quad \kappa^*[S^2]^*=e_0e_3,
 \end{gathered}
\]
where $e_1,e_2$ come from $E$, and juxtaposition denotes exterior
product. The image is the subring generated by $e_0,e_1,e_2$ and
$e_0e_3$, which is exactly the invariant subring of
$\bigwedge^*\Z\langle e_0,e_1,e_2,e_3\rangle$.
The top class has coefficient one. Together with
\eqref{fill:rankone-cohomology}, this proves that specialization is
an integral isomorphism in every degree.

At both points the representative $e$ has valuation zero.
The monomial change preserves this valuation, so $e$ lies
in the chart of the zero vertex and extends.
\end{proof}

\subsection{The rank-two Mumford construction}

At $\alpha_-$, use a local parameter $q$ and write
\[
 \begin{gathered}
 X_q\simeq(\CC^*)^2/\langle h,g\rangle,\\
 h(z,v)=(c_1qz,c_2v),\qquad
 g(z,v)=(c_3z,\lambda c_4qv),
 \end{gathered}
\]
where the $c_i=c_i(q)$ are holomorphic units. The simple zero of
$\Phi_0$ adds a factor $q$ to the second multiplier of $g$.
Thus the exponent vectors are $(1,0)$ and $(0,1)$.
We use the same periodic fan construction as in
Section~\ref{fill:rankone-section}, now with two independent
degenerating directions. This is the local model of
\cite[Proposition~2.7]{Engel}, with $\lambda$ absorbed into the
fourth unit multiplier.

\begin{proposition}\label{fill:a2}
The punctured family at $\alpha_-$ admits a proper filling with
smooth total space. Its central fiber $W$ has normalization
$\overline W=dP_6$, with opposite boundary curves identified.
The section $e$ extends across this filling, and
specialization induces integral isomorphisms
\[
 \specialize^k:H^k(W;\Z)\xrightarrow{\sim}H^k(X_q;\Z)^T,
 \qquad 0\leq k\leq4.
\]
\end{proposition}

\begin{proof}
Apply \cite[Proposition~2.8]{Engel}, using the periodic triangulation
obtained by dividing each unit square along its antidiagonal.
As in Section~\ref{fill:rankone-section}, regular cones give smooth
charts and the period quotient gives the filling. Here the charts
are $q=z_0z_1z_2$, and quotienting by $\Z^2$ identifies the opposite
boundary curves of one $dP_6$.

The representative $e$ has valuation $(0,1)$. Translating it by
$g^{-1}$ gives valuation zero, so $e$ extends in
the zero-vertex chart, as in \cite[Proposition~2.9]{Engel}.

For $k=1,2,3$, the integral specialization isomorphism is the
Mumford-fiber assertion of \cite[Lemma~7.4]{Engel}. That part of
the proof uses only the unimodular period lattice and the toric
stratification, so it applies to this local model independently
of the other singular fibers. Degree zero is immediate. In degree
four, the collapse has one orientation-preserving sheet away from
the double curves and hence sends the top generator to the top
generator, as in the rank-one case.
\end{proof}

\subsection{Base change for the logarithmic transformation}

Write $s=1/t$ near infinity, and let $\Delta$ be the disc with coordinate
$r$ on the cubic cover $s=r^3$. Put $\omega=e^{2\pi i/3}$.

\begin{proposition}\label{fill:good}
After the cubic base change at infinity, $X_U$ extends over $r=0$
to a smooth proper family of complex two-tori $X'\to\Delta$.
The central fiber is an abelian surface. The section $e$ extends as
the zero section, and the deck group acts on $X'$ preserving $e$.
\end{proposition}

\begin{proof}
The cubic cover gives good reduction for a fiber of type $IV^*$
\cite[Ch.~V, \S10]{BHPV}. In \eqref{init:weierstrass}, put
$x'=r^2x$ and $y'=r^3y$. The smooth elliptic model $S'\to\Delta$ is given by
\[
 (y')^2=(x')^3-3r(1+r^3)x'+2+3r^3+2r^6.
\]
The section $P'=(r^2,\sqrt2)$ meets the central fiber at the nonzero
flex $(0,\sqrt2)$, so $P'(0)$ has order three.

Set $M'=\cO_{S'}(P'-O')$. The lift $F$ extends to $(M')^\times$ by
the descent argument of \cite[Proposition~3.1]{Engel}. On an equivariant
common resolution \cite{BM}, the two line bundles differ by a vertical
divisor. At the generic point of the component mapping onto $S'_0$,
their local frames thus differ by a power of $r$ times a unit.
Translation by $6P'$ preserves that component and fixes
$r$, so the power cancels in the linearization. The linearization and
its inverse then extend over the remaining points of $S'$ by Hartogs.
Write $F'$ for the resulting lift of translation by $-6P'$.

It remains to specify the identification that preserves $e$. Since $P$
and $O$ are disjoint, $M|_O$ is the conormal line of $O$. The local
parameters at the two zero sections satisfy
\[
 -\frac{s^2x}{s^3y}=r^{-2}\left(-\frac{x'}{y'}\right).
\]
Thus, under the natural identification of the line bundles over the
punctured disc, $r^2e$ extends to a nowhere-zero section $e'$ of
$M'|_{O'}$. We use the identification
\begin{equation}
 v\longmapsto r^2v,\qquad e\longmapsto e'.
 \label{fill:normalized-comparison}
\end{equation}
This rescaling commutes with the lifted translation because $r$ is a
base coordinate.

With $e'$ as identity, form $X'=(M')^\times/\langle F'\rangle$.
For $|\lambda|$ sufficiently small, the contraction argument of
Proposition~\ref{init:open} gives a smooth proper family of complex two-tori.
On $S'_0$, the bundle $M'_0$ has order three and $F'_0$ acts by a
scalar, since translation by $6P'(0)$ is the identity. A degree-three isogeny trivializes
$M'_0$, so $X'_0$ is isogenous to a product of elliptic curves and
is an abelian surface \cite[Corollary~3.2]{Engel}.

The deck map $(r,x',y')\mapsto(\omega^{-1}r,\omega x',y')$ preserves
$P'-O'$. Its lift to $M'$ is uniquely normalized by requiring it to
preserve $e'$. Over the punctured disc this is the original deck action
under \eqref{fill:normalized-comparison}. Hence its cube is the identity
and it commutes with $F'$ there, and therefore on all of $\Delta$.
It descends to the required group action on $X'$.
\end{proof}

Here is a little more information about the extended deck transformation. Let $\rho$ denote the deck transformation over $r\mapsto\omega^{-1}r$.
On $S'_0$ it sends $(x',y')$ to $(\omega x',y')$, so it multiplies the
local parameter $-x'/y'$ at the origin of $S'_0$ by $\omega$. It fixes the elliptic
kernel of $X'_0\to S'_0$ pointwise because its lift to the line over
that point is linear and fixes $e'(0)\ne0$. The lift of $\rho|_{X'_0}$ fixing the origin is therefore given in eigenvector coordinates on the universal
cover $\CC^2$ by
\[
 (z,w)\longmapsto(\omega z,w).
\]

\subsection{The logarithmic transformation}

Let $\Lambda=H_1(X'_0;\Z)$. A positive loop about $s=0$ lifts from
$r$ to $\omega r$, and $\rho$ identifies the endpoint with the starting
fiber. Thus $\widehat T_*=\rho_*$ is the monodromy on $\Lambda$.
Fix an invariant period $\ell$ and an invariant integral covector $\psi$
such that
\begin{equation}
 \widehat T_*\ell=\ell,\qquad \psi\circ\widehat T_*=\psi,\qquad
 \psi(\ell)=1.
 \label{fill:twist-data}
\end{equation}
The choice of $\ell$ in \eqref{main:twist} satisfies these conditions
by \eqref{mon:matrices}. Over $\Delta$ the period lattice is constant,
so $\ell/3$ defines a $\rho$-invariant section of order three.

\begin{proposition}\label{fill:log}
Twist the deck action by this translation, as in
\cite[Construction~3.5]{Engel}:
\[
 \widetilde\rho(x)=\rho(x)+\ell/3.
\]
This action is free. The quotient $X'/\langle\widetilde\rho\rangle$
is a smooth proper filling of the original punctured family, with
central fiber $3B$. Its reduction $B=X'_0/\langle\widetilde\rho\rangle$
is bielliptic. Under parallel transport in $X'$, specialization from
$B$ is the pullback of the covering $p:X'_0\to B$.
\end{proposition}

\begin{proof}
The translation by $\ell/3$ commutes with $\rho$, so
$\widetilde\rho^3=1$. Off the central fiber the base action is free.
On $X'_0$ the covector $\psi$ induces a map to $\R/\Z$ under which
$\widetilde\rho$ acts by addition of $\psi(\ell)/3=1/3$.
Thus neither $\widetilde\rho$ nor its square fixes a point.
The quotient map is unramified, so $s=r^3$ gives multiplicity three
along $B$. The quotient is smooth and proper. Its reduced fiber $B$
is covered by the abelian surface $X'_0$;
the linear part of the action is nontrivial on the elliptic curve
$S'_0$ and trivial on the elliptic kernel. Hence $B$ is bielliptic
\cite[Ch.~V, \S5]{BHPV}.

To identify the punctured families, use the translation of
\cite[Lemma~3.6]{Engel}. For $r\ne0$, put
\[
 \sigma(r)=\frac{\log r}{2\pi i}\ell.
\]
Changing the logarithm adds a period, so $\sigma$ is a single-valued
holomorphic section of $X'$ over the punctured disc. Our deck generator
sends $r$ to $\omega^{-1}r$, so
$\sigma(\omega^{-1}r)=\rho(\sigma(r))-\ell/3$. For $x\in X'_r$ this gives
\begin{equation}
 \widetilde\rho(x)+\sigma(\omega^{-1}r)=\rho\bigl(x+\sigma(r)\bigr).
 \label{fill:log-conjugacy}
\end{equation}
Thus $x\mapsto x+\sigma(r)$ identifies the twisted quotient with the
original family over the punctured disc. The comparison
\eqref{fill:normalized-comparison} preserves $e$, so $\sigma$ is the
entire translation in the gluing map.

Finally, average an Ehresmann connection on $X'$ over the twisted
action. Radial parallel transport retracts $X'$ equivariantly onto
$X'_0$. On a nearby fiber the descended retraction is the covering
$p$ after parallel transport, so specialization is $p^*$.
\end{proof}

Let $D_2,D_+,D_-,D_*$ be small closed discs about
$0,\alpha_+,\alpha_-,\infty$, respectively. Write $X_{D_i}\to D_i$
for the local filling and $m_i=e_*[\partial D_i]$ for the lift of the
positive boundary loop of $D_i$ along $e$. The monodromy on the fiber lattice
$\Lambda$ is $\widehat T_i=(T_i^{-1})^{\mathsf t}$.

\begin{lemma}
\label{fill:boundary}
Each inclusion $j_i:X_{\partial D_i}\hookrightarrow X_{D_i}$ induces a surjection
\[
 (j_i)_*:\pi_1(X_{\partial D_i})\twoheadrightarrow\pi_1(X_{D_i}).
\]
For the three discs about $0,\alpha_+,\alpha_-$, we have
\[
 m_i\in\ker (j_i)_*,\qquad
 \im(\widehat T_i-I)\subseteq\ker (j_i)_*.
\]
Here $\im(\widehat T_i-I)$ is primitive in $\Lambda$.
For each $\eta\in\Lambda$, let $t_\eta$ denote translation by $\eta$
on the universal cover of $X'_0$. At $\infty$, the group $\pi_1(X_{D_*})$
is generated by the translations $t_\eta$ defined by the elements $\eta\in \Lambda$ and the meridian $m_*$,
with defining relations
\begin{equation}
 \begin{aligned}
 t_\eta t_{\eta'} &= t_{\eta+\eta'} && \forall\eta,\eta'\in\Lambda,\\
 m_*t_\eta m_*^{-1} &= t_{\widehat T_*^{-1}\eta} && \forall\eta\in\Lambda,\\
 m_*^3 &= t_{-\ell}.
 \end{aligned}
 \label{fill:log-group}
\end{equation}
\end{lemma}

\begin{proof}
For the rank-one Mumford construction, the local total space retracts
onto $W$, which is homotopy equivalent to $E\times(S^2\vee S^1)$ by
\eqref{fill:rankone-collapse}. The nearby fiber supplies its three
fundamental-group generators, and the pinched primitive circle dies.
The extension of $e$ supplies a disc bounding the lifted meridian.

For the rank-two construction, a chart $q=z_0z_1z_2$ has vanishing lattice
$\{(m_i)\in\Z^3:\sum m_i=0\}$, with basis
$(1,-1,0),(0,1,-1)$. The corresponding loops, for example
$(\epsilon e^{i\theta},\epsilon e^{-i\theta},c)$, bound discs obtained
by multiplying all three coordinates by $a\in[0,1]$. These discs
remain in the total space over the disc. Since the triangles are
unimodular, their edge differences form an integral basis of $\Z^2$.
The phase-lattice identification therefore identifies these circles
with the full primitive image of $\widehat T_i-I$.

Inclusion of the dense torus into a toric chart is surjective on
$\pi_1$: the toric divisors impose their ray relations, and higher
codimension strata add no generators. Taking the periodic quotient
adds only the period loops, already present in the boundary. This
proves boundary surjectivity. As before, $e$ extends over the
disc, so the lifted meridian is null-homotopic.

For the logarithmic transformation, Proposition~\ref{fill:log} identifies
the fundamental group with the Bieberbach group of $B$.
On the universal cover of $X'_0$, take the affine lift $b$ with
translation part $\ell/3$. Then $b^3=t_\ell$ and
$bt_\eta b^{-1}=t_{\widehat T_*\eta}$ for all $\eta\in\Lambda$. Under
\eqref{fill:log-conjugacy}, the section $e$ is represented
on the twisted cover by $-\sigma(r)$. Along a positive base loop it
ends at $\widetilde\rho^{-1}(-\sigma(r))$. Therefore $m_*=b^{-1}$,
giving \eqref{fill:log-group}. The fiber loops and this meridian
generate the group, proving the remaining surjectivity assertion.
\end{proof}

Choose pairwise disjoint discs about the four special values and
smaller concentric closed discs. Glue the local fillings to
$X_U(\lambda)$ over the resulting annuli, using the identifications
just constructed. There are no triple overlaps among the local
pieces. The local maps are proper and separated and the transition
maps are biholomorphic, so analytic descent gives a proper map with
smooth total space
\[
 f:X_2(\lambda)\longrightarrow\PP^1.
\]
Thus $X_2(\lambda)$ is a compact complex threefold. We now compute
its topology from the monodromy and the local maps just described.

\section{Main result}\label{main:section}

To specify the logarithmic transformation, fix a fiber over $U$.
Write its first cohomology and homology lattices as $V$ and $\Lambda=V^\vee$,
using the basis $(\psi,u,w,\delta)$ of Section~\ref{mon:section} and
hats for the dual basis. Take
\begin{equation}
 \ell=\widehat\psi-2\widehat u-4\widehat w.
 \label{main:twist}
\end{equation}
The matrices in \eqref{mon:matrices} give
\[
 \widehat T_*\ell=\ell,\qquad \psi(\ell)=1,
 \qquad \widehat T_*=(T_*^{-1})^{\mathsf t}.
\]
Thus the construction of Section~\ref{init:initial}, with translation
\(\ell/3\) on the family of Proposition~\ref{fill:good}, defines a threefold
\(X=X_2(\lambda)\) over \(\PP^1\).

\begin{theorem}\label{main:theorem}
For every sufficiently small nonzero \(\lambda\), the threefold
\(X_2(\lambda)\) is compact, smooth, and simply connected, and
\begin{equation}
 H_k(X_2(\lambda);\Z)\simeq
 \begin{cases}
 \Z,&k=0,6,\\
 \Z/2,&k=2,3,\\
 0,&k=1,4,5.
 \end{cases}
 \label{main:homology}
\end{equation}
In particular, it is a complex rational homology six-sphere.
\end{theorem}

Smoothness and compactness follow from the local constructions.
We compute the monodromy in Section~\ref{mon:section}, prove simple
connectivity in Section~\ref{pi:section}, and compute integral homology
in Section~\ref{hom:section}. The calculation keeps track of the
integral specialization maps and yields transgression coefficients
$(1,2,2,1)$, which account for the two copies of $\Z/2$.

\section{The monodromy representation}
\label{mon:section}

We first determine the monodromy of the elliptic surface, then track
how a logarithm of $P$ changes around each singular value, and use
these changes to compute the monodromy of the torus family.
The resulting matrices will be used to compute
the fundamental group and integral homology. All matrices act on
column vectors, and all meridians are positive.

\subsection{The elliptic factorization}

\begin{proposition}
\label{mon:elliptic}
Fix a smooth fiber $S_{t_0}$ of the elliptic surface $\pi:S\to\PP^1$.
One can choose positive meridians based at $t_0$ around
$\infty,0,\alpha_+,\alpha_-$ and a symplectic basis of $H_1(S_{t_0};\Z)$
so that the corresponding monodromy matrices are
\begin{equation}
\begin{alignedat}{2}
 A_*&=\begin{pmatrix}-1&-1\\1&0\end{pmatrix},
 &\qquad A_2&=\begin{pmatrix}1&2\\0&1\end{pmatrix},\\
 A_+&=\begin{pmatrix}2&1\\-1&0\end{pmatrix},
 &\qquad A_-&=\begin{pmatrix}1&1\\0&1\end{pmatrix}.
\end{alignedat}
\label{mon:elliptic-matrices}
\end{equation}
They satisfy $A_*A_2A_+A_-=I$ and $A_*^3=I$.
\end{proposition}

The proof is given in Appendix~\ref{app:elliptic}.

We label the last $I_1$-value $\alpha_-$ and choose the simple zero
of $\Phi_0$ there.  The other $I_1$-value is $\alpha_+$.

\subsection{The periods of the section $P$}

A local logarithm of $P$ is a holomorphic lift of $P_t$ from the elliptic
curve $S_t$ to its universal cover $\CC$. Continuing this lift around a
loop can change it by an integral period. For the four meridians of
Proposition~\ref{mon:elliptic}, write $k_*,k_2,k_+,k_-\in\Z^2$ for these
changes in the chosen period basis. All $k_i$ are column vectors.

\begin{lemma}\label{mon:normal-lemma}
One can choose the initial logarithm of $P$ so that its period changes are
\begin{equation}
 k_*=\binom{1}{-1},\qquad k_2=\binom{1}{0},\qquad k_+=k_-=0,
 \label{mon:normal}
\end{equation}
after reversing both vectors of the elliptic basis if necessary.
\end{lemma}

\begin{proof}
Near an $I_n$ fiber, a multiplicative representative of $P$ is a power
of the local parameter times a holomorphic unit. Its logarithm therefore
changes by a multiple of the primitive vanishing period. By
\eqref{mon:elliptic-matrices}, write
\[
 k_2=m(1,0),\qquad k_+=n(1,-1),\qquad k_-=p(1,0).
\]
Going around all four punctures gives
\[
 k_*+A_*k_2+A_*A_2k_++A_*A_2A_+k_-=0,
\]
so $k_*=(m-2n-p,-m+n)$. Changing the initial logarithm changes each
$k_i$ by $(I-A_i)r$ for some common $r=(r_1,r_2)\in\Z^2$. Thus
\[
 (m,n,p)\longmapsto(m-2r_2,\ n-r_1-r_2,\ p-r_2).
\]
Taking $r_2=p$ and $r_1=n-p$ gives $(d,0,0)$, where $d=m-2p$.
This is the unique representative with $n=p=0$.

We must show that $P$ gives $d=\pm1$. Let $\mathbb V$ be the sheaf of
integral elliptic periods on $U$, and $j:U\hookrightarrow\PP^1$.
Let $\mathcal E^0$ be the sheaf of local holomorphic sections meeting
the component containing $O$ in every fiber. The fiberwise exponential
map gives
\begin{equation}
 0\longrightarrow j_*\mathbb V
 \longrightarrow\cO_{\PP^1}(-1)
 \xrightarrow{\exp}\mathcal E^0\longrightarrow0.
 \label{mon:exponential}
\end{equation}
Here $\cO_{\PP^1}(-1)$ is the bundle of tangent lines along $O$.
Locally, the exponential is surjective onto the identity component,
and its kernel consists of the single-valued periods. This also
describes the sequence at the singular fibers.

Both $H^0$ and $H^1$ of $\cO_{\PP^1}(-1)$ vanish. Thus taking
logarithms identifies the global sections of $\mathcal E^0$ with
$H^1(\PP^1,j_*\mathbb V)$. The latter consists of the period changes,
modulo changes of logarithm, for which $k_i\in(A_i-I)\Z^2$ at every
puncture. For the representative
$(d,0,0)$, the conditions at $0$ and $\infty$ are
\[
 d(1,0)\in2\Z(1,0),\qquad
 (d,-d)\in\begin{pmatrix}-2&-1\\1&-1\end{pmatrix}\Z^2.
\]
They are equivalent to $2\mid d$ and $3\mid d$, so this group is
$6\Z$ in the coordinate $d$.

By Lemma~\ref{init:height}, all sections are multiples of $P$, and $P$
generates the component groups of orders two and three. The sections
meeting the identity component everywhere are therefore exactly
$6\Z P$. The preceding isomorphism sends $6P$ to $\pm6$, hence sends
$P$ to $d=\pm1$. Reversing both elliptic basis vectors preserves every
$A_i$ and changes the sign of every $k_i$. Choose $d=1$.
\end{proof}

\subsection{The monodromy of the torus family}

\begin{lemma}\label{mon:jacobi-lemma}
There is an integral basis $(\psi,u,w,\delta)$ of
$V=H^1(X_{t_0};\Z)$ with the following property. If a loop has elliptic
monodromy $A=\left(\begin{smallmatrix}a&b\\c&d\end{smallmatrix}\right)$
and changes the logarithm of $P$ by $k=(k_1,k_2)$, its monodromy on
$V$ has the form
\begin{equation}
 T=
 \begin{pmatrix}
  1&6(ck_1-ak_2)&6(dk_1-bk_2)&c_0\\
  0&a&b&k_1\\
  0&c&d&k_2\\
  0&0&0&1
 \end{pmatrix},
 \qquad c_0\in\Z.
 \label{mon:jacobi}
\end{equation}
In this basis the class $\xi=u\wedge w+6\psi\wedge\delta$ is
monodromy invariant.
\end{lemma}

\begin{proof}
On a simply connected part of $U$, write
$S_t=\CC/(\Z+\tau\Z)$ and $P_t=[p]$. Using the additive coordinate
$z$ on $\CC$ and writing the line coordinate as $v=e^{-2\pi iy}$,
the gluing of $M=\cO(P-O)$ gives three periods
\[
 a_0=(0,1),\qquad a_1=(1,0),\qquad a_2=(\tau,-p).
\]
The lift $F$ is a group translation covering $-6P$. Its logarithm
supplies the fourth period
\begin{equation}
 a_3=(-6p,\kappa),
 \label{mon:periods}
\end{equation}
where $\kappa$ is the additive change in $y$ under $F$.
These four periods form an integral basis of $H_1(X_t;\Z)$.
In the dual basis, set
\begin{equation}
 \psi=a_3^\vee,\qquad u=-a_2^\vee,\qquad
 w=a_1^\vee,\qquad\delta=-a_0^\vee.
 \label{mon:dual-basis}
\end{equation}

We first check that $\xi$ is independent of the choices. As an
alternating form on the period lattice, its only nonzero basic
pairings are
\[
 \xi(a_1,a_2)=1,\qquad \xi(a_0,a_3)=6.
\]
Replacing $p$ by $p+m+n\tau$ changes the line coordinate by
$y_{\mathrm{old}}=y_{\mathrm{new}}+nz$, up to translation, and gives
\[
 a_1\longmapsto a_1+na_0,\qquad
 a_2\longmapsto a_2-ma_0,\qquad
 a_3\longmapsto a_3-6ma_1-6na_2\pmod{\Z a_0}.
\]
These changes preserve the displayed pairings. Changes of elliptic
basis are generated by
$(a_1,a_2)\mapsto(a_1,a_2+a_1)$ and
$(a_1,a_2)\mapsto(a_2,-a_1)$, with $a_0,a_3$ fixed.
Changing the logarithm of $F$ adds an integral multiple of $a_0$ to
$a_3$. These transformations also preserve $\xi$, so it is
monodromy invariant.

The line period $a_0$ is fixed, while $a_3$ is fixed modulo
$\Z a_0+\Z a_1+\Z a_2$. Thus the diagonal blocks on
$(\psi,u,w,\delta)$ are $1,A,1$. The change $p\mapsto p+m+n\tau$
computed above gives $\delta\mapsto\delta+mu+nw$ modulo $\Z\psi$.
Consequently the middle entries in the last column are $k_1,k_2$.
If the first row is $(1,r_1,r_2,c_0)$, expanding $T\xi=\xi$ gives
\[
 r_1b-r_2a+6k_1=0,\qquad r_1d-r_2c+6k_2=0.
\]
Since $ad-bc=1$, these imply
$r_1=6(ck_1-ak_2)$ and $r_2=6(dk_1-bk_2)$, proving
\eqref{mon:jacobi}.
\end{proof}

\begin{proposition}
\label{mon:representation}
In the marking above, the geometric monodromies on $V$ are
\begin{equation}
\begin{alignedat}{2}
 T_*&=\begin{pmatrix}
 1&0&-6&-2\\0&-1&-1&1\\0&1&0&-1\\0&0&0&1
 \end{pmatrix},
 &\qquad T_2&=\begin{pmatrix}
 1&0&6&3\\0&1&2&1\\0&0&1&0\\0&0&0&1
 \end{pmatrix},\\[1ex]
 T_+&=\begin{pmatrix}
 1&0&0&0\\0&2&1&0\\0&-1&0&0\\0&0&0&1
 \end{pmatrix},
 &\qquad T_-&=\begin{pmatrix}
 1&0&0&-1\\0&1&1&0\\0&0&1&0\\0&0&0&1
 \end{pmatrix}.
\end{alignedat}
\label{mon:matrices}
\end{equation}
They satisfy $T_*T_2T_+T_-=I$, $T_*^3=I$, and $T_i\xi=\xi$.
\end{proposition}

\begin{proof}
Substituting \eqref{mon:elliptic-matrices} and \eqref{mon:normal} into
\eqref{mon:jacobi} leaves only the upper-right entries
$c_*,c_2,c_+,c_-$ to determine.

By Proposition~\ref{fill:good}, $T_*^3=I$. The upper-right entry of
$T_*^3-I$ is $3c_*+6$, so $c_*=-2$.
At $I_2$, the nonzero columns of $T_2-I$ are $2(3\psi+u)$ and
$c_2\psi+u$. Since $T_2-I$ has rank one by
Proposition~\ref{fill:rankone}, $c_2=3$.
The same rank condition at $\alpha_+$ gives $c_+=0$.

With these three entries fixed, $T_*T_2T_+T_--I$ has only one
possibly nonzero entry, namely $c_-+1$ in the upper-right corner.
The sphere relation therefore gives $c_-=-1$. The invariance of $\xi$
was proved in Lemma~\ref{mon:jacobi-lemma}.
\end{proof}

\section{The fundamental group}\label{pi:section}

The homological monodromy is $\widehat T_i=(T_i^{-1})^{\mathsf t}$
on $\Lambda=V^\vee$. The section $e$ of Proposition~\ref{init:section}
splits the fundamental group sequence over $U$ and lifts the positive
base meridians to loops $m_i$. Reading paths from right to left gives
\begin{equation}
\begin{split}
 \pi_1(X_U)=\langle\Lambda,m_*,m_2,m_+,m_-\mid\;&
 [\Lambda,\Lambda]=1,\\
 &m_i x m_i^{-1}=\widehat T_i^{-1}x,\\
 &m_-m_+m_2m_*=1\rangle.
\end{split}
\label{pi:presentation}
\end{equation}
Indeed, the corresponding product of conjugation actions is
\[
 \widehat T_-^{-1}\widehat T_+^{-1}
 \widehat T_2^{-1}\widehat T_*^{-1}=I.
\]
In expressions involving
\(\Lambda\), we use additive notation for its abelian subgroup.

\begin{lemma}\label{pi:coinvariants}
The three Mumford fillings leave at most the infinite cyclic group
generated by \(\widehat\psi\). The image of \(\ell\) in this
group is \(\widehat\psi\).
\end{lemma}
\begin{proof}
Direct inverse transposition of \eqref{mon:matrices} gives
\begin{align*}
 \im(\widehat T_2-I)&=\Z(2\widehat w+\widehat\delta),\\
 \im(\widehat T_+-I)&=\Z(\widehat u+\widehat w),\\
 \im(\widehat T_--I)&=\Z\widehat w\oplus\Z\widehat\delta.
\end{align*}
Consequently
\[
 \Lambda\Big/\sum_{i=2,+,-}\im(\widehat T_i-I)
 =\Z[\widehat\psi],\qquad [\ell]=[\widehat\psi].
\]
The section $e$ extends across each of these fillings, so
\(m_2=m_+=m_-=1\). Their boundary maps are surjective on
fundamental groups by Lemma~\ref{fill:boundary}. Van Kampen therefore
introduces no new generators and imposes the displayed relations,
possibly together with further relations.
\end{proof}

\begin{proposition}\label{pi:trivial}
The threefold \(X\) is simply connected.
\end{proposition}
\begin{proof}
The three Mumford fillings kill $m_2,m_+,m_-$, so the sphere
relation in \eqref{pi:presentation} gives $m_*=1$.
By Lemma~\ref{fill:boundary}, $m_*^3=t_{-\ell}$, so the fiber loop
represented by $\ell$ is trivial in $\pi_1(X)$. By
Lemma~\ref{pi:coinvariants}, this kills the remaining fiber
generator. Since all boundary maps are surjective on fundamental
groups, no further generators occur, and $\pi_1(X)=1$.
\end{proof}

\section{The integral homology}\label{hom:section}

Retain $\xi=u\wedge w+6\psi\wedge\delta$ and put
\[
 \phi=\xi\wedge\psi=\psi\wedge u\wedge w,\qquad
 \nu=\psi\wedge u\wedge w\wedge\delta.
\]
Brackets denote coinvariant classes for the monodromy over $U$. As before, $j:U\hookrightarrow\PP^1$;
let $H=c_1(\cO_{\PP^1}(1))$ be the positive generator of $H^2(\PP^1;\Z)$.

\begin{lemma}\label{hom:lattices}
The monodromy invariants \(H^0(U,R^qf_*\Z)\) and the
monodromy coinvariants of \(H^q(X_{t_0};\Z)\) are free of rank one
for \(0\leq q\leq4\). In increasing degree, their generators are
\[
 1,\quad\psi,\quad\xi,\quad\phi,\quad\nu,
\]
and
\[
 1,\quad[\delta],\quad[\psi\wedge\delta],\quad[u\wedge w\wedge\delta],\quad[\nu],
\]
respectively.
Moreover, \([u\wedge w]=6[\psi\wedge\delta]\) and \([\xi]=12[\psi\wedge\delta]\).
\end{lemma}
\begin{proof}
Use lexicographic exterior bases induced by
\((\psi,u,w,\delta)\). In degree one, the columns of
\(T_--I\) kill \([u]\) and \([\psi]\), and those of
\(T_+-I\) kill \([w]\). No relation involves a nonzero multiple
of \([\delta]\). The same two matrices force a common invariant
to be a multiple of \(\psi\).

In degree two, write the exterior basis as
\[
(\psi u,\psi w,\psi\delta,uw,u\delta,w\delta),
\]
temporarily suppressing wedge signs. The \(T_-\) and \(T_+\)
relations kill the first, second, fifth, and sixth classes. The
\(T_2\) relations give \([uw]=6[\psi\delta]\), and \(T_*\)
adds no further relation. Thus the quotient is free of rank one.
The invariant equations force the same four coordinates to vanish
and the third coordinate to be six times the fourth, giving
\(\Z\xi\).

In degree three, \(T_-\) and \(T_+\) force the last three
coordinates of a common invariant to vanish. Their coinvariant
relations, together with those of \(T_2\), kill the first three
basis classes integrally. The fourth remains free. Finally,
\(\det T_i=1\), which settles degree four.
\end{proof}

\begin{lemma}\label{hom:finite-specialization}
Let \(B\) be the reduced fiber over infinity.
For \(q=1,2,3,4\), the specialization map
\[
 H^q(B;\Z)\longrightarrow H^q(X_{t_0};\Z)^{T_*}
\]
is injective, and its image has index \(3,1,1,3\), respectively.
The groups \(H^q(B;\Z)\) are torsion-free.
\end{lemma}
\begin{proof}
Let \(p:X'_0\to B\) be the degree-three quotient map in
Proposition~\ref{fill:log}, which identifies specialization with
\(p^*\). The cohomology of the multiple fiber is that of its
reduction \(B\).

Using \(\widehat T_*=(T_*^{-1})^{\mathsf t}\), we obtain
\[
 H^1(X_{t_0};\Z)^{T_*}=\Z\psi\oplus\Z h,
 \quad h=-2u+w-3\delta,
 \quad \Lambda/(\widehat T_*-I)\Lambda\simeq\Z^2,
\]
and \(h(\ell)=0\). Abelianizing \eqref{fill:log-group} with $b=m_*^{-1}$
gives
\[
 H_1(B;\Z)=
 \frac{\bigl(\Lambda/(\widehat T_*-I)\Lambda\bigr)\oplus\Z b}{\langle3b-[\ell]\rangle}
 \simeq\Z^2.
\]
Here \([\ell]\) is primitive because \(\psi(\ell)=1\).
The Euler characteristic of \(B\) is zero, as it is a finite
free quotient of a four-torus. Poincar\'e duality and the universal
coefficient theorem give \(b_2(B)=2\) and show that its integral
cohomology is torsion-free in every degree.

A covector \(a\psi+bh\) descends precisely when its value on
\(\ell\) is divisible by three. Hence
\[
 p^*H^1(B;\Z)=\Z(3\psi)\oplus\Z h.
\]
In degree two the invariant basis \((\psi h,\xi)\) has
intersection matrix
\[
 \begin{pmatrix}0&-3\\-3&12\end{pmatrix}.
\]
The pullback of the unimodular rank-two intersection lattice of
\(B\) has determinant of absolute value \(3^2=9\). This equals
the determinant of the full invariant lattice, so the index is one.

In degree three an invariant basis is
\[
 \phi,\qquad
 -4\psi u\delta+2\psi w\delta+uw\delta.
\]
Its pairing with \((\psi,h)\) is
\(
\left(\begin{smallmatrix}0&1\\3&0\end{smallmatrix}\right)
\).
If \(I_1,I_3\) are the two pullback indices, comparison with the
downstairs perfect pairing gives \(I_1I_3=3\). Since \(I_1=3\),
we have \(I_3=1\). In top degree a degree-three oriented covering
pulls back by multiplication by three.
\end{proof}

\begin{lemma}\label{hom:parabolic}
For \(q=1,2,3\),
\[
 H^1\!\left(\PP^1,j_*\bigl((R^qf_*\Z)|_U\bigr)\right)=0.
\]
\end{lemma}
\begin{proof}
For an integral local system \(L\), the parabolic complex is
\begin{equation}
 L\xrightarrow{d^0}\bigoplus_i(T_i-I)L
   \xrightarrow{d^1}L,
 \label{hom:parabolic-complex}
\end{equation}
where
\begin{align*}
 d^0x&=((T_i-I)x)_i,\\
 d^1(y_*,y_2,y_+,y_-)
 &=y_*+T_*y_2+T_*T_2y_++T_*T_2T_+y_-.
\end{align*}
Since \(d^1d^0=0\), this is a complex. Its middle cohomology is
\(H^1(\PP^1,j_*L)\); its last cokernel is the coinvariant group.
To see this, cut the punctured sphere
along its marked meridians and attach the four discs with their
invariant stalks.

For \(L=(R^qf_*\Z)|_U\), the ranks of the four local images are
\[
 (2,1,1,2),\quad(4,2,2,2),\quad(2,1,1,2)
\]
in degrees \(1,2,3\). Lemma~\ref{hom:lattices} gives ranks
\(3,5,3\) for both \(d^0\) and \(d^1\). The middle cohomology
is therefore rationally zero.

The remaining issue is torsion. Consider the integer matrix
\[
 D_q=\begin{pmatrix}
 \bigwedge^qT_*-I\\\bigwedge^qT_2-I\\
 \bigwedge^qT_+-I\\\bigwedge^qT_--I
 \end{pmatrix}.
\]
Write $D_q[R,C]$ for the submatrix with row set $R$ and column
set $C$, using one-based indices in the lexicographic exterior bases.
The following maximal-rank minors are units:
\[
\begin{aligned}
 \det D_1[\{2,3,6\},\{2,3,4\}]&=1,\\
 \det D_2[\{1,2,4,6,7\},\{1,2,3,5,6\}]&=-1,\\
 \det D_3[\{1,3,5\},\{2,3,4\}]&=-1.
\end{aligned}
\]
Thus \(\im D_q\) is saturated even in the ambient direct sum
\((H^q(X_{t_0};\Z))^{\oplus4}\). The quotient of the middle lattice
in \eqref{hom:parabolic-complex} by \(\im d^0\) embeds in this
torsion-free ambient quotient. Its subgroup
\(\ker d^1/\im d^0\) is torsion-free and rationally zero, hence zero.
\end{proof}

\begin{proposition}\label{hom:leray-page}
The \(E_2\)-page of the Leray spectral sequence of
\(f:X\to\PP^1\) with integral coefficients is
\begin{equation}
\begin{array}{c|c|c|c}
q&H^0(\PP^1,R^qf_*\Z)&H^1(\PP^1,R^qf_*\Z)
 &H^2(\PP^1,R^qf_*\Z)\\ \hline
4&3\Z\nu&0&\Z H[\nu]\\
3&\Z\phi&0&\Z H[u\wedge w\wedge\delta]\\
2&\Z\xi&0&\Z H[\psi\wedge\delta]\\
1&3\Z\psi&0&\Z H[\delta]\\
0&\Z&0&\Z H.
\end{array}
\label{hom:page}
\end{equation}
The first column is identified with its image under restriction to
\(X_{t_0}\). Here \(H[\alpha]\) denotes the class in
\(H^2(\PP^1,R^qf_*\Z)\) corresponding to the monodromy
coinvariant \([\alpha]\). All other entries are zero.
\end{proposition}
\begin{proof}
Proper base change identifies the stalk of \(R^qf_*\Z\) at a
special value with the cohomology of the reduced fiber. The natural
map to \(j_*\bigl((R^qf_*\Z)|_U\bigr)\) is the specialization map. At the three
Mumford fibers it is an integral isomorphism by
Propositions~\ref{fill:rankone} and~\ref{fill:a2}. At infinity it
is the injective pullback of Lemma~\ref{hom:finite-specialization}.
The only cokernels are the skyscraper sheaves $(\Z/3)_\infty$ in
degrees one and four. Hence
\begin{align*}
 R^0f_*\Z&=\Z,\\
 0\longrightarrow R^1f_*\Z&\longrightarrow j_*\bigl((R^1f_*\Z)|_U\bigr)
       \longrightarrow(\Z/3)_\infty\longrightarrow0,\\
 R^qf_*\Z&=j_*\bigl((R^qf_*\Z)|_U\bigr)\quad(q=2,3),\\
 0\longrightarrow R^4f_*\Z&\longrightarrow\Z\nu
       \longrightarrow(\Z/3)_\infty\longrightarrow0.
\end{align*}
All specialization maps are injective, so there are no additional
kernels supported at the special points. The global sections $\psi$
and $\nu$ surject onto the two skyscraper quotients. The associated
long exact sequences therefore leave $H^1$ zero and replace the
invariant generators by $3\psi$ and $3\nu$.
Lemmas~\ref{hom:lattices} and~\ref{hom:parabolic} give the table.
There are no columns above degree two because the base has real dimension two.
\end{proof}

\begin{lemma}\label{hom:clutch}
For \(1\leq q\leq4\), the Leray differential
\[
 d_2:H^0(\PP^1,R^qf_*\Z)
 \longrightarrow H^2(\PP^1,R^{q-1}f_*\Z)
\]
is given, up to a common sign, by
\begin{equation}
 d_2(a)=\frac{H}{3}[\iota_\ell a],
 \label{hom:contraction}
\end{equation}
where \(\iota_\ell\) denotes contraction with the cycle \(\ell\)
of \eqref{main:twist}. Division by three is taken in the monodromy
coinvariants of \(H^{q-1}(X_{t_0};\Z)\).
\end{lemma}
\begin{proof}
On the cubic cover, the logarithmic transformation is identified
with the original family by translation through
\(\sigma(r)=(\log r)/(2\pi i)\,\ell\). A positive turn in
\(s=r^3\) changes this translation by \(\ell/3\), modulo periods.

In the filtered \v Cech complex with fiber cochains, the prism
homotopy for a translation through a cycle \(z\) lowers exterior
degree by one and induces contraction \(\iota_z\). The change of
this homotopy around the overlap circle represents the Leray
transgression. Pairing the oriented disc with the base fundamental
class gives \eqref{hom:contraction}.

The three Mumford identifications preserve $e$,
so their affine translation cocycles vanish. The good-reduction
comparison also preserves $e$. Thus the logarithmic
translation supplies the entire transgression. We check integrality
on the generators of \eqref{hom:page} below.
\end{proof}

\begin{proposition}\label{hom:differentials}
The four nonzero differentials on the \(E_2\)-page are
\begin{equation}
\begin{aligned}
 d_2(3\psi)&=H,&
 d_2(\xi)&=2H[\delta],\\
 d_2(\phi)&=2H[\psi\wedge\delta],&
 d_2(3\nu)&=H[u\wedge w\wedge\delta].
\end{aligned}
\label{hom:four-differentials}
\end{equation}
\end{proposition}
\begin{proof}
Contraction with $\ell$ gives
\begin{align*}
 \iota_\ell\psi&=1,\\
 \iota_\ell\xi&=4u-2w+6\delta,\\
 \iota_\ell\phi&=uw+2\psi w-4\psi u,\\
 \iota_\ell\nu&=uw\delta+2\psi w\delta-4\psi u\delta.
\end{align*}
In the coinvariants, these become $1$, $6[\delta]$,
$6[\psi\wedge\delta]$, and $[u\wedge w\wedge\delta]$, respectively, by
Lemma~\ref{hom:lattices}. Applying
\eqref{hom:contraction} to $3\psi,\xi,\phi,3\nu$
gives the four asserted integral differentials.
\end{proof}

\begin{proof}[Proof of Theorem~\ref{main:theorem}]
The local construction gives a compact smooth complex threefold,
and Proposition~\ref{pi:trivial} gives simple connectivity. By
\eqref{hom:page} and \eqref{hom:four-differentials}, all four
differentials are injective and have cokernels
\(0,\Z/2,\Z/2,0\). There are no differentials with \(r\ge3\).
The only nonzero terms of $E_\infty$ are
\[
 E_\infty^{0,0}=\Z,\quad
 E_\infty^{2,1}=\Z/2,\quad
 E_\infty^{2,2}=\Z/2,\quad
 E_\infty^{2,4}=\Z.
\]
Each total degree has a single nonzero filtration quotient, so no
extension problem remains. Therefore
\[
 H^k(X;\Z)\simeq
 \begin{cases}
 \Z,&k=0,6,\\
 \Z/2,&k=3,4,\\
 0,&k=1,2,5.
 \end{cases}
\]
Poincar\'e duality for the complex orientation,
\(H_k(X;\Z)\simeq H^{6-k}(X;\Z)\), proves
\eqref{main:homology}.
\end{proof}

\section{Topological type}\label{top:section}

Let $M_1$ be the oriented smooth six-manifold obtained by surgery on
the conic
\[
 C_0=\{[s^2:st:t^2:0]:[s:t]\in\PP^1\}\subset\PP^3.
\]
An oriented normal framing identifies a tubular neighborhood of $C_0$
with $S^2\times D^4$, and the surgery replaces it by $D^3\times S^3$:
\begin{equation}\label{top:model}
 M_1=\bigl(\PP^3\setminus\operatorname{int}(S^2\times D^4)\bigr)
       \mathop{\cup}_{S^2\times S^3}(D^3\times S^3).
\end{equation}
The boundary gluing is induced by the framing, which exists and is
unique up to homotopy. Thus \eqref{top:model} specifies one oriented
diffeomorphism type.

\begin{theorem}\label{top:theorem}
For every sufficiently small nonzero $\lambda$, there is an
orientation-preserving diffeomorphism $X_2(\lambda)\cong M_1$,
where $X_2(\lambda)$ carries its complex orientation. Moreover,
\begin{equation}\label{top:cube}
 \langle a^3,[X]_2\rangle=1,
 \qquad 0\ne a\in H^2(X;\mathbb F_2).
\end{equation}
\end{theorem}

There are exactly two oriented diffeomorphism types of closed simply
connected spin six-manifolds with homology \eqref{main:homology},
distinguished by the value of $\langle a^3,[X]_2\rangle$.
The proof is given in Appendix~\ref{app:topology}.

\appendix
\section{The elliptic surface}\label{app:elliptic}

We prove the facts used in Section~\ref{init:initial}, using the
standard theory of elliptic surfaces and the height pairing
\cite{Miranda,Shioda}.
We then compute the elliptic monodromy of Proposition~\ref{mon:elliptic}.

\begin{proof}[Proof of Proposition~\ref{init:surface}]
Homogenize \eqref{init:weierstrass} with fundamental line bundle
$\cO_{\PP^1}(1)$. For $A=-3(1+t)$ and $B=2+3t+2t^2$,
\[
 \Delta=-16(4A^3+27B^2)
       =-432t^2(4t^2+8t+5).
\]
The coefficient $A$ is nonzero at the three finite zeros of $\Delta$,
which therefore give fibers $I_2,I_1,I_1$. At infinity, put $s=t^{-1}$,
$X=s^2x$, and $Y=s^3y$. The local coefficients are
\[
 A_\infty=-3s^3(1+s),\qquad
 B_\infty=s^4(2+3s+2s^2).
\]
Their vanishing orders, together with that of the discriminant, are
$(3,4,8)$. The equation is minimal and the fiber is $IV^*$ by Tate's
algorithm. There are no other singular fibers. Finally,
$\chi(\cO_S)=1$, and the canonical bundle formula gives
$K_S\simeq\pi^*\cO_{\PP^1}(-1)$; hence this elliptic surface with
section is rational~\cite{Miranda,BHPV}.
\end{proof}

\begin{proof}[Proof of Lemma~\ref{init:height}]
The reducible fibers have root lattices $E_6$ and $A_1$.
Since $\rho(S)=10$, the Shioda--Tate formula gives
$\rk\MW(S)=10-2-6-1=1$. For every nonzero section $R$,
\[
 \langle R,R\rangle
 =2+2(R\cdot O)-\operatorname{contr}_\infty(R)
                    -\operatorname{contr}_0(R)
 \geq 2-\frac43-\frac12=\frac16.
\]
Here the local contributions belong to $\{0,4/3\}$ and $\{0,1/2\}$,
respectively~\cite{Shioda}. In particular, $\MW(S)$ is torsion-free.

Substitution verifies \eqref{init:explicit-section}. This section is
disjoint from $O$, including at infinity, where its coordinates are
$(X,Y)=(s^2,\sqrt2\,s^2)$. At $0$ and $\infty$ it passes through
the singular point of the Weierstrass fiber, so on the minimal regular
model it meets a nonidentity component. Indeed, the smooth locus of
the Weierstrass fiber is the identity component of the N\'eron model.
Thus its local contributions are $1/2$ and $4/3$, and
\[
 \langle P,P\rangle=2-\frac43-\frac12=\frac16.
\]
If $P=nR$ for a generator $R$, then
$1/6=n^2\langle R,R\rangle\geq n^2/6$, so $n=\pm1$.
\end{proof}

\begin{proof}[Proof of Lemma~\ref{init:linearization}]
Set $Q=6P$. Write $N=\mathcal{H}om(t_Q^*M,M)$. Translation acts trivially on
$\Pic^0$ of the generic fiber, so $N$ is represented by a vertical
divisor. Since $t_Q$ preserves every component, $N$ has multidegree
zero on each fiber. The fiber intersection matrices and the absence
of multiple fibers imply that $N\simeq\pi^*L$ for a line bundle
$L$ on $\PP^1$.

All local height corrections involving the narrow section ${6P}$ vanish.
Hence
\[
 6=\langle Q,Q\rangle=2+2(Q\cdot O),\qquad
 1=\langle P,Q\rangle=1+(Q\cdot O)-(P\cdot Q),
\]
so $Q\cdot O=P\cdot Q=2$. Restricting $N=M\otimes t_Q^*M^{-1}$
to $O$ and using $O^2=-1$ gives
\[
 \deg L=(P-O)\cdot O-(P-O)\cdot Q=1-(2-2)=1.
\]
Therefore $L\simeq\cO_{\PP^1}(1)$.
\end{proof}

\begin{proof}[Proof of Proposition~\ref{mon:elliptic}]
For primitive $v=(r,s)^{\mathsf t}\in\Z^2$, put
$v\cdot w=\det(v,w)$.  Our positive twist convention is
\begin{equation}
 \tau_v(w)=w+(v\cdot w)v,\qquad
 \tau_v=\begin{pmatrix}1-rs&r^2\\-s^2&1+rs\end{pmatrix}.
 \label{mon:twist}
\end{equation}
In particular, $\tau_{-v}=\tau_v$.

Choose paths from $t_0$ to the three finite singular values. Let
$a,b,c\in H_1(S_{t_0};\Z)$ be the primitive vanishing cycles of the
$I_2,I_1,I_1$ fibers, respectively, transported to $S_{t_0}$ along
these paths. The two nodes of the $I_2$ fiber have parallel vanishing
cycles on the nearby torus, so each contributes the same positive
Dehn twist $\tau_a$ and $A_2=\tau_a^2$. Each $I_1$ fiber has one node,
giving $\tau_b$ and $\tau_c$, respectively, by Picard--Lefschetz.
Let $B=\tau_a^2\tau_b\tau_c$ be their
ordered product. The sphere relation gives $B=A_*^{-1}$, so
$\operatorname{tr}(B)=-1$ because the fiber at infinity has type $IV^*$.

We now show that a change of oriented integral basis and Hurwitz moves
reduce any such triple $(\tau_a^2,\tau_b,\tau_c)$ to
\begin{equation}
 \bigl(\tau_{(1,0)}^2,\ \tau_{(1,-1)},\ \tau_{(1,0)}\bigr).
 \label{mon:standard-factorization}
\end{equation}
Here a Hurwitz move replaces two adjacent factors $(A,B)$ by
$(B,B^{-1}AB)$ or vice versa, preserving their product.
Geometrically, these moves change the paths to the singular values.

Set $x=a\cdot b$, $y=b\cdot c$, and $z=c\cdot a$.  Expanding
the rank-one operators in \eqref{mon:twist} gives
\[
 \operatorname{tr}(\tau_a^2\tau_b\tau_c)
 =2-2x^2-y^2-2z^2+2xyz.
\]
Thus
\begin{equation}
 2x^2+y^2+2z^2-2xyz=3.
 \label{mon:markov}
\end{equation}
If $xyz<0$ with all entries nonzero, the left side is at least
$7$.  The case $y=0$ is excluded by parity.  If $x=0$, then
$|y|=|z|=1$; if $z=0$, then $|x|=|y|=1$.

Suppose all entries are nonzero. Choose orientations so that
$x,y,z>0$, and by symmetry assume $x\geq z$. The Vieta replacements
\[
 \bar y=2xz-y,\qquad \bar x=yz-x
\]
preserve \eqref{mon:markov}, and
\[
 y\bar y=2x^2+2z^2-3,\qquad
 x\bar x=(y^2+2z^2-3)/2.
\]
If $z=1$, the equation is $x^2+(y-x)^2=1$, giving
$(x,y,z)=(1,1,1)$.

Suppose $z\geq2$. Equality $x=z$ would give
\[
 (y-x^2)^2=(x^2-2)^2-1,
\]
strictly between the consecutive squares $(x^2-3)^2$ and
$(x^2-2)^2$; hence $x>z$.  If $y>xz$, replacing $y$ by
$\bar y$ strictly decreases the sum of the coordinates while
keeping them positive.  Equality $y=xz$ would imply
$(x^2-2)(z^2-2)=1$, which is impossible.  Otherwise $y$ is the
smaller root of
\[
 f(Y)=Y^2-2xzY+2x^2+2z^2-3,
\]
and
\[
 f(2x/z)=1+(x^2-z^2)(4/z^2-2)<0.
\]
Consequently $y<2x/z$ and $0<\bar x<x$.

Repeated descent, interchanging $x,z$ when necessary, ends at
$(1,1,1)$. At this triple the
$y$-replacement fixes the triple, whereas the $x$- and
$z$-replacements produce zero.  Since the replacements are
involutions, there is no distinct positive predecessor.  Thus every
nonzero solution has $|x|=|y|=|z|=1$ and $xyz=1$.

For this last case, choose signs so that $x=y=z=1$.  A change of
basis and a power of $\tau_a$ put
\[
 a=(1,0),\qquad b=(0,1),\qquad c=(-1,-1).
\]
The Hurwitz move $(b,c)\mapsto(c,\tau_c^{-1}b)$ has
$\tau_c^{-1}b=-a$.  Conjugating next by $\tau_a^{-2}$ fixes $a$
and sends $c$ to $(1,-1)$, giving
\eqref{mon:standard-factorization}.

If $z=0$, primitivity gives
$c=\pm a$, and the same centralizer gives the stated form directly.
If $x=0$, a Hurwitz move on the last two factors replaces the
intersection triple by $(-z,-y,yz-x)$, whose entries are nonzero.
This reduces to the preceding case.

Applying \eqref{mon:twist} to the three factors in
\eqref{mon:standard-factorization} gives $A_2,A_+,A_-$ in
\eqref{mon:elliptic-matrices}. Their product is
$B=\left(\begin{smallmatrix}0&1\\-1&-1\end{smallmatrix}\right)$,
so $A_*=B^{-1}$ is the remaining matrix.
\end{proof}

\section{Proof of the topological classification}\label{app:topology}

We first recall the classification needed for
Theorem~\ref{top:theorem}. For a closed simply connected spin
six-manifold $N$ with homology \eqref{main:homology}, the even-degree
mod-two cohomology has one generator in each of degrees $0,2,4,6$.
Writing these generators as $1,a,d,z$, duality gives $ad=z$, and the
only undetermined product is
\[
 a^2=\varepsilon d,\qquad
 \varepsilon=\langle a^3,[N]_2\rangle\in\mathbb F_2.
\]
Moreover, $p_1(N)=2q_1(N)=0$, since the spin characteristic class
$q_1(N)$ lies in $H^4(N;\Z)=\Z/2$. Andrus's classification
\cite[Theorem~5.2]{Andrus}, with $n=2$, therefore determines the
oriented diffeomorphism type by $\varepsilon$: its invariants are
the even-degree $\Z/n$ cohomology ring and $p_1$, with an additional
Pontryagin cubing operation only when $3\mid n$.
See also~\cite[Theorems~1 and~7]{Zhubr}.

Both values occur. Attach a framed three-handle to $S^2\times D^5$
along an embedded sphere representing twice the generator of
$H_2(S^2\times S^4)$. The boundary of the resulting handlebody has homology
\eqref{main:homology} and is simply connected and spin. The
degree-two mod-two class extends over the handlebody, so its cube
evaluates to zero on the boundary.

The value one is realized by $M_1$. The complex normal bundle of $C_0\subset\PP^3$
is $\cO_{\PP^1}(4)\oplus\cO_{\PP^1}(2)$. Its underlying oriented
real bundle is trivial, since oriented rank-four bundles on $S^2$
are classified by $w_2$ and here $w_2=6\bmod2=0$. The normal framing
is unique up to homotopy because $\pi_2(SO(4))=0$. Surgery preserves
simple connectivity and the spin condition, and the surgery homology
sequence gives \eqref{main:homology} for $M_1$.
The mod-two hyperplane class extends over the surgery trace because
the conic has even degree. Its restriction to $M_1$ is nonzero, and
bordism invariance gives
\[
 \langle a_{M_1}^3,[M_1]_2\rangle
   =\langle\bar H^3,[\PP^3]_2\rangle=1.
\]
Thus \eqref{top:model} is the oriented diffeomorphism type with value one.

For $X=X_2(\lambda)$, Theorem~\ref{main:theorem} gives
$H^2(X;\Z)=0$, hence $c_1(X)=0$ and $X$ is spin. It remains to
compute \eqref{top:cube}. We retain the marking
\eqref{mon:dual-basis}, the invariant class $\xi=uw+6\psi\delta$,
and the logarithmic period $\ell$ of \eqref{main:twist}.
Let $D\subset\CC$ be a closed disc containing the three finite
special values, and write $X_D=f^{-1}(D)$.

\begin{lemma}\label{top:fixed}
The scalar involution $\iota(z,v)=(z,-v)$ extends to $X$.
Its fixed set consists of an elliptic curve $E$ over $0$ and a
rational curve $C$ over $\alpha_-$. Moreover, $\pi_1(X_D)=\Z$.
\end{lemma}
\begin{proof}
The involution commutes with the periods, extends over the Mumford
charts, and descends through the logarithmic quotient. On a smooth
fiber the contraction estimate excludes $-v=F^m v$ for $m\ne0$,
and $m=0$ is impossible.

At $0$, the coordinates
$(x,y)=(v,v^2/z)$ give $(x,y)\mapsto(-x,y)$, fixing the double
elliptic curve. In its two normal directions the action is $-1$.
At $\alpha_+$ the action on the remaining elliptic factor is a
nonzero two-torsion translation, so there are no fixed points.

In the rank-two Mumford construction at $\alpha_-$, each triangle
has two odd normal weights and one zero tangential weight. The fixed
curves correspond to vertical edges of the periodic triangulation.
There is one such edge modulo the period lattice, with endpoints
at the two distinct triple points; it gives one smooth $\PP^1$.
At infinity the involution acts freely on the good-reduction torus,
and the commuting quotient of odd order three creates no fixed points.

Lemma~\ref{pi:coinvariants} shows that $\pi_1(X_D)$ is a quotient
of $\Z[\widehat\psi]$. By Propositions~\ref{fill:rankone}
and~\ref{fill:a2} and the degree-one Leray sequence over the disc $D$,
the invariant class $\psi$ extends to $X_D$. It evaluates as one
on $\widehat\psi$, so there are no further relations.
\end{proof}

Put $G=\langle\iota\rangle\simeq\Z/2$, and let $a$ denote the
nonzero class in $H^2(X;\mathbb F_2)$.

\begin{lemma}\label{top:line}
There is a $G$-linearized complex line bundle $\mathcal L$ on $X_D$
such that
\[
 \rho_2c_1(\mathcal L)=a|_{X_D},\qquad
 \langle c_1(\mathcal L),[C]\rangle\equiv1\pmod2.
\]
The stabilizer acts trivially on $\mathcal L|_E$ and nontrivially
on $\mathcal L|_C$.
\end{lemma}
\begin{proof}
Rigidify along the section $e$. As in Lemma~\ref{fill:tate},
write $h_1$ for the elliptic period and $h_2$ for a multiplicative
representative of $P$; let $R$ be the fiber multiplier of $F$.
The periods $h,F$ lift to the auxiliary line, with coordinate $\zeta$,
by the symmetric factors
\begin{equation}\label{top:line-factors}
 \begin{split}
 h(z,v,\zeta)&=(h_1z,h_2v,-z^{-1}\zeta),\\
 F(z,v,\zeta)&=(h_2^{-6}z,Rv,h_2^{-3}R^3v^6\zeta).
 \end{split}
\end{equation}
They commute and give first Chern class $\xi$ on each smooth fiber.
The scalar in the second line is fixed by symmetry: for the inversion
lift
\[
 J(z,v,\zeta)=(z^{-1},v^{-1},-z^{-1}\zeta),
\]
a factor $Bv^6$ satisfies $JFJ=F^{-1}$ precisely when
$B^2=h_2^{-6}R^6$. We choose $B=h_2^{-3}R^3$.

These symmetric line factors~\cite{BL} patch under the elliptic
markings. Indeed, their quadratic refinement modulo two is
\[
 Q(\gamma)=u(\gamma)+w(\gamma)+u(\gamma)w(\gamma)
             \pmod2,
\]
the odd elliptic theta characteristic, zero on the radical spanned
by $a_0,a_3$. All four monodromies \eqref{mon:matrices} preserve it.
Rigidification removes the scalar ambiguity; the remaining sign
choices are degree-zero two-torsion factors. Since $v^6$ is even in
$v$, the lift $(z,v,\zeta)\mapsto(z,-v,\zeta)$ commutes with
\eqref{top:line-factors}.

At $0$, the orders of $(h_1,h_2,R)$ are $(2,1,-3)$. Replacing $F$ by
$g=Fh^3$, as in Lemma~\ref{fill:tate}, changes its line factor to a
unit times $z^{-3}v^6=y^3$, while the factor for $h$ is
$-z^{-1}=-x^{-2}y$. An integral support function on the rank-one
fan is $-k(k-1)$, up to a linear term in the nondegenerating
direction. Its successive slopes are $-2k$, so the local frames
have even $x$-weight. Thus the character along $E$ is trivial.

At $\alpha_-$, the orders are $(1,0,1)$ and the monomial line
factors are $z^{-1}$ and $t^3v^6$. On the periodic triangulation
an integral support function is
\begin{equation}\label{top:support}
 \varphi(m,n)=-\frac{m(m-1)}2+3n^2.
\end{equation}
Its period differences are $-m$ and $6n+3$. The $v$-slope on
either triangle of a unit square is odd, giving the nontrivial
character along $C$. Successive $z$-slopes differ by one, so the
degree on $C$ is odd. Translating the representative of $e$ from
the vertex $(0,1)$ to the zero vertex changes the $v$-slopes by
multiples of six and preserves both parities. At $\alpha_+$ the
same factors extend over the rank-one fan.

To identify the mod-two class globally, use the origin-preserving
comparison \eqref{fill:normalized-comparison} on the cubic cover
at infinity. The symmetric line extends over the good-reduction
disc. Two symmetric rigidified lines with the same alternating
form differ by a two-torsion degree-zero line; its integral Chern
class on the boundary $T^4\times S^1$ is zero. Taking the tensor
norm for the degree-three cover gives descent under the untwisted
deck action without changing the mod-two class.

The logarithmic clutching changes this class on the boundary by
contraction with $\ell$, as in Lemma~\ref{hom:clutch}. Here
\begin{equation}\label{top:even-clutch}
 \iota_\ell\xi=4u-2w+6\delta=-2(-2u+w-3\delta),
\end{equation}
which is even in the full local lattice. Hence the mod-two classes
agree on the boundary after the logarithmic transformation.
Transfer for the degree-three cover shows that they agree downstairs
as well. Mayer--Vietoris gives a global class restricting to
$\rho_2c_1(\mathcal L)$ on $X_D$. Its restriction to a smooth fiber
is $\bar\xi=uw\ne0$, so it is $a$.
\end{proof}

\begin{proof}[Proof of Theorem~\ref{top:theorem}]
We compute the cube by mod-two equivariant localization, writing
$\tau$ for the generator of $H^1(BG;\mathbb F_2)$.
First, $a$ lifts to $\widetilde a\in H_G^2(X;\mathbb F_2)$.
Indeed, $H^1(X;\mathbb F_2)=0$ eliminates the first possible
obstruction in the Borel spectral sequence. The remaining obstruction
lies in $H^3(BG;\mathbb F_2)$ and vanishes because a fixed point
gives a section over $BG$.

We compare this lift on $X_D$ with the equivariant reduction of
$c_1(\mathcal L)$. Their difference vanishes in ordinary cohomology.
By Lemma~\ref{top:fixed} and the scalar circle isotopy,
\[
 \pi_1((X_D)_G)=\Z\times\Z/2.
\]
The low-degree sequence for the simply connected universal cover
shows that the difference comes from
$H^2(B(\Z\times\Z/2);\mathbb F_2)$. Every fixed point induces the
same homomorphism $\Z/2\to\Z\times\Z/2$, since the latter group
has a unique element of order two. The difference therefore has
one common coefficient of $\tau^2$ on the fixed set. Subtracting
that multiple of $\tau^2$ from $\widetilde a$, Lemma~\ref{top:line}
gives
\[
 \widetilde a|_C=a|_C+\tau^2,\qquad
 \widetilde a|_E=a|_E+\tau\eta,
 \quad\eta\in H^1(E;\mathbb F_2).
\]

For either fixed curve $F=C,E$, the complex normal bundle has
$w_2(N_F)=\rho_2(c_1(X)|_F-c_1(TF))=0$, and $\iota$ is $-1$ in
both normal directions. Thus its equivariant Euler class is
$e_G(N_F)=\tau^4$. Localization gives
\[
 \langle a^3,[X]_2\rangle
   =\sum_{F=C,E}\int_F\frac{(\widetilde a|_F)^3}{\tau^4}.
\]
This follows from the equivariant Thom isomorphism after inverting
$\tau$, since the action is free away from the fixed curves.
On the oriented surface $E$, the products
$(a|_E)^2$, $a|_E\eta$, and $\eta^2$ vanish, so its contribution
is zero. The rational curve contributes
\[
 \int_C\frac{(a|_C+\tau^2)^3}{\tau^4}
    =\int_C a|_C=1
\]
by Lemma~\ref{top:line}. This proves \eqref{top:cube}, and the
classification above identifies $X$ with $M_1$.
\end{proof}


\begin{thebibliography}{99}

\bibitem{AlbaneseMilivojevic}
M.~Albanese and A.~Milivojevi\'c,
On the minimal sum of Betti numbers of an almost complex manifold,
\emph{Differential Geom. Appl.} \textbf{62} (2019), 101--108.
\href{https://doi.org/10.1016/j.difgeo.2018.10.002}{doi:10.1016/j.difgeo.2018.10.002}.

\bibitem{Andrus}
D.~G.~Andrus, \emph{Classification of Certain 6-Manifolds},
M.Sc. thesis, McMaster University, 1977.

\bibitem{BHPV}
W.~Barth, K.~Hulek, C.~Peters, and A.~Van de Ven,
\emph{Compact Complex Surfaces}, second ed., Springer, Berlin, 2004.

\bibitem{BM}
E.~Bierstone and P.~D.~Milman,
Canonical desingularization in characteristic zero by blowing up the
maximum strata of a local invariant,
\emph{Invent. Math.} \textbf{128} (1997), 207--302.

\bibitem{BL}
C.~Birkenhake and H.~Lange,
\emph{Complex Abelian Varieties}, second ed., Springer, Berlin, 2004.

\bibitem{BorelSerre}
A.~Borel and J.-P.~Serre,
Groupes de Lie et puissances r\'eduites de Steenrod,
\emph{Amer. J. Math.} \textbf{75} (1953), no.~3, 409--448.
\href{https://doi.org/10.2307/2372495}{doi:10.2307/2372495}.

\bibitem{AlpogeClaude}
Claude and L.~Alp\"oge,
\emph{A compact complex threefold fibred by tori over the projective line,
and the six-sphere}, manuscript, 2026.
\url{https://alpo.ge/s6.pdf} (accessed 10 September 2026).

\bibitem{Clemens}
C.~H.~Clemens, Degeneration of K\"ahler manifolds,
\emph{Duke Math. J.} \textbf{44} (1977), 215--290.

\bibitem{Engel}
P.~Engel, \emph{Complex structures on $S^6$}, manuscript,
13 September 2026.
\url{https://philip-engel.github.io/S6.pdf}
(accessed 3 November 2026).

\bibitem{Miranda}
R.~Miranda, \emph{The Basic Theory of Elliptic Surfaces},
ETS Editrice, Pisa, 1989.

\bibitem{Mumford}
D.~Mumford, An analytic construction of degenerating abelian varieties
over complete rings,
\emph{Compositio Math.} \textbf{24} (1972), 239--272.

\bibitem{Shioda}
T.~Shioda, On the Mordell--Weil lattices,
\emph{Comment. Math. Univ. St. Pauli} \textbf{39} (1990), 211--240.

\bibitem{Zhubr}
A.~V.~Zhubr, Classification of simply connected six-dimensional spin
manifolds, \emph{J. Soviet Math.} \textbf{10} (1978), no.~3, 451--453.
\href{https://doi.org/10.1007/BF01476852}{doi:10.1007/BF01476852}.
Translated from \emph{Zap. Nauchn. Sem. LOMI} \textbf{45} (1974), 71--74.
\end{thebibliography}
\end{document}